\documentclass[12pt]{amsart}
\usepackage{latexsym}
\usepackage{amsmath}
\usepackage{amsfonts}
\usepackage{amsthm}
\usepackage{mathtools}
\usepackage{amssymb, bbm}
\usepackage{amsrefs}
\usepackage[T1]{fontenc}

\numberwithin{equation}{section}

\title{On a generalized conjecture by Alzer and Matkowski}
\author{W\l{}odzimierz Fechner}%\footnote{corresponding author, wlodzimierz.fechner@p.lodz.pl}
\address{W\l{}odzimierz Fechner\\ Institute of Mathematics, Lodz University of Technology, al. Politechniki 8, 93-590 \L\'od\'z, Poland}
\email{wlodzimierz.fechner@p.lodz.pl}
\author{Marta Pierzcha\l{}ka}
\address{Marta Pierzcha\l{}ka \\ Institute of Mathematics, Lodz University of Technology, al. Politechniki 8, 93-590 \L\'od\'z, Poland}
\email{247887@edu.p.lodz.pl}
\author{Gabriela Smejda}
\address{Gabriela Smejda \\ Institute of Mathematics, Lodz University of Technology, al. Politechniki 8, 93-590 \L\'od\'z, Poland}
\email{247899@edu.p.lodz.pl}
\newtheorem{thm}{Theorem}%[section]
\newtheorem{cor}{Corollary}%[section]
\theoremstyle{remark}
\newtheorem{rem}{Remark}%[section]
\newtheorem{ex}{Example}%[section]
\theoremstyle{definition}
\newcommand{\R}{\mathbb{R}}

\newcommand{\K}{\mathbb{K}}
\newcommand{\C}{\mathbb{C}}

\keywords{Cauchy difference, biadditive functional}
\subjclass[2020]{Primary 39B22; Secondary 39B32}

\begin{document}

\baselineskip=17pt

\begin{abstract}
We study a recent conjecture proposed by Horst Alzer and Janusz Matkowski concerning a bilinearity property of the Cauchy exponential difference for real-to-real functions. The original conjecture was affirmatively resolved by Tomasz Ma\l{}olepszy. We deal with generalizations for real or complex mappings acting on a linear space.
\end{abstract}

\maketitle 

\section{Introduction}

Recently H. Alzer and J. Matkowski \cite{AM} have studied the following functional equation: 
\begin{equation}\label{main0}
f(x+y) = f(x)f(y) - \alpha x y, \quad x, y \in \R,
\end{equation}
where $\alpha \in \R$ is a non-zero parameter and $f\colon\R\to\R$ is an unknown function. They proved two theorems on equation \eqref{main0}. The first result with a short proof \cite{AM}*{Theorem 1} completely describes solutions 
of \eqref{main0} in case  $f$ has a zero. More precisely, they showed that if $f$ solves \eqref{main0} and it has a zero, then $\alpha > 0$ and
either $f(x) = 1 -\sqrt{\alpha}x$, or $f(x) = 1 + \sqrt{\alpha}x$ for $x \in \R$.
The second theorem with a longer proof \cite{AM}*{Theorem 2} provides the solutions to equation \eqref{main0} under the assumption that $f\colon\R\to\R$ is differentiable at least at one point. In this case, there are the same two solutions (clearly, both have a zero).
In \cite{AM} the authors formulated the following conjecture:

\bigskip

\noindent
\textsc{Conjecture (Alzer and Matkowski)}. Every solution $f\colon\R\to\R$ of \eqref{main0} has a zero.

\bigskip

This conjecture has been answered affirmatively by T. Ma\l{}olepszy, see \cite{MM}.
In the present note, we will determine the solutions of a more general equation, namely 
\begin{equation}\label{main}
f(x+y) = f(x)f(y) -  \phi(x, y), \quad x, y \in X,
\end{equation}
where $X$ is a linear space over the field $\K \in \{ \R, \C\}$,  $\phi\colon X \times X \to\K$ is a biadditive functional and $f\colon X\to\K$ is a function.
The motivation for such a generalization comes from an article by K. Baron and Z. Kominek \cite{BK}, in which the authors, in connection with a problem proposed by S. Rolewicz \cite{R}, studied mappings defined on a real linear space with the additive Cauchy difference bounded from below by a bilinear functional.

\section{Main results}

In this section, it is assumed that $X$ is a linear space over $\K \in \{ \R, \C\}$,  $\phi\colon X \times X \to\K$ is a biadditive functional and $f\colon X\to\K$.
We will consider two situations, depending on the behavior of the biadditive functional $\phi$ on the diagonal.

\begin{thm}\label{t1}
Assume that $\phi$ and $f$ solve \eqref{main} and
\begin{equation}
\label{z0}
\exists_{z_0 \in X} \phi(z_0,z_0)\neq 0.
\end{equation}
Then there exists a constant $a \in \K\setminus\{0\}$ such that
\begin{equation}
\label{f}
f(x) = a\phi(x, z_0) + 1, \quad x \in X
\end{equation}
and moreover
\begin{equation}
\label{xx}
a^2\phi(x, z_0)^2=\phi(x,x) , \quad x \in X.
\end{equation}
\end{thm}
\begin{proof}
Substituting $y = z_0$ and then $y = -z_0$ in \eqref{main} we obtain
$$f(x+z_0) = f(x)f(z_0) - \phi(x, z_0), \quad x \in X$$
and
$$f(x-z_0) = f(x)f(-z_0) + \phi(x, z_0), \quad x \in X.$$
Replace $x$ by $x+z_0$ in the latter formula and join it with the former one to arrive at 
\begin{align*}
    f(x) &= f(x+z_0)f(-z_0) + \phi(x+z_0,z_0) \\
    &= [f(x)f(z_0) - \phi(x, z_0)]f(-z_0) + \phi(x,z_0)+\phi(z_0,z_0), \quad x \in X.
\end{align*}
Denote $c \coloneqq f(z_0)$,  $d\coloneqq f(-z_0)$ and $\beta = \phi(z_0,z_0)\neq 0$. We get
$$(1 - cd) f(x)= (1-d)\phi(x, z_0) + \beta, \quad x \in X.$$
As stated in the proof of Theorem 1 in \cite{AM}*{Theorem 1}, $f(0)=1$ (the argument works in our case, as well). Therefore, from \eqref{main} applied for $x=z_0$ and $y=-z_0$ we deduce
$$1 = f(z_0-z_0) = f(z_0)f(-z_0)  + \beta,$$
thus $1 - cd = \beta$. Since $\beta \neq 0$, then denoting  $a\coloneqq (1-d)/\beta$ we get \eqref{f}. The case $a=0$ is impossible, since it leads to a contradiction with \eqref{z0}.

To prove equality \eqref{xx} apply \eqref{main} with substitution $y=-x$ and obtain
$$ f(x)f(-x) = 1 - \phi(x,x), \quad x \in X. $$
Now, use the already proven formula \eqref{f} to derive \eqref{xx} after some reductions.
\end{proof}

\begin{rem}\label{r1}
From Theorem \ref{t1}, we see that under assumption \eqref{z0} and with a fixed functional $\phi$  there are always either no solutions or exactly two solutions $f$ of \eqref{main}. Indeed, if $f$ is a solution, then it must be of the form \eqref{f} with some constant $a \in \K$. Substituting this into \eqref{main} leads us to the equality:
$$
a^2\phi(x,z_0)\phi(y,z_0) = \phi(x, y),  \quad x, y \in X,
$$ 
which is true for two different values of $a\neq 0$.
Therefore, in case there do exist solutions, functional $\phi$ is of the form
$$\phi(x,y) = a^2 F(x)\cdot F(y), \quad x, y \in X$$
with an additive, nonzero functional $F\colon X \to \K$, and the two possible functions $f$ are given by \eqref{f}.
\end{rem}

We have the following corollary in the real case.

\begin{cor}\label{cR}
Assume that $\K = \R$
%\begin{equation}\label{z1}
%\forall_{z \in X} \phi(z,z) \leq 0
%\end{equation}
and
\begin{equation}
\label{z2}
\exists_{z_0 \in X} \phi(z_0,z_0)< 0.
\end{equation}
Then equation \eqref{main} has no solutions.
\end{cor}
\begin{proof}
Inequality \eqref{z2} implies that condition \eqref{z0} holds true. Then, from Theorem \ref{t1} we obtain formula \eqref{xx}. However, in the real case formula \eqref{xx} implies that $\phi(z_0,z_0)>0$, which leads to a contradiction.
\end{proof}

In the complex case, every element of the field has a complex root of second order. Therefore, we can state the next corollary.

\begin{cor}\label{cC}
Assume that $\K = \C$, $\phi$ and $f$ solve \eqref{main}, $\phi$ satisfies \eqref{z0} and $w\colon X\to\C$ is a map such that $$w^2(x)= \phi(x,x), \quad x \in X.$$
Then  $$f(x) = w(x) +1, \quad x \in X.$$ 
\end{cor}

The next theorem deals with the remaining case for $\phi$ and is easy to prove.

\begin{thm}\label{t2}
Assume that $\phi$ and $f$ solve \eqref{main} and
\begin{equation}
\label{z}
\forall_{z \in X} \phi(z,z) = 0.
\end{equation}
Then $\phi = 0$ on $X \times X$ and 
\begin{equation}
\label{f2}
f(x+y) = f(x)f(y), \quad x, y \in X.
\end{equation}
Consequently, either $f=0$ or there exists an additive functional $A\colon X \to \K$ such that $f = \exp \circ A$.
\end{thm}
\begin{proof}
It suffices to apply a well-known result, which states that if a multiadditive function vanishes on a diagonal, then it vanishes everywhere, cf. \cite{Kuczma}*{Corollary 15.9.1, p. 448}. The final part follows from the form of solutions of the exponential Cauchy equation cf. \cite{Kuczma}*{Theorem 13.1.1, p. 343}.
\end{proof}

The following corollary is immediate and offers an alternative proof of the conjecture by Alzer and Matkowski.

\begin{cor}[T. Ma\l{}olepszy]\label{c3}
Assume that $\alpha \in \R$ and $f\colon \R\to\R$ solves \eqref{main0}. Then $\alpha\geq 0$ and moreover, in case $\alpha > 0$
either $f(x) = 1 -\sqrt{\alpha}x$, or $f(x) = 1 + \sqrt{\alpha}x$ for $x \in \R$. Conversely, both mappings solve \eqref{main0}.
\end{cor}
\begin{proof}
Firstly, substituting $\alpha = 0$ in \eqref{main0} we obtain the exponential Cauchy's equation, for which the solutions are known.
Now assume that $\alpha \neq 0$, $\K = \R$, $X = \R$ and $\phi(x,y) \coloneqq \alpha xy$. Let $z_0 = 1$. Then $\beta = \phi(1,1) = \alpha \neq 0$. From \eqref{f} we have
$$f(x) = a \phi(x,1) + 1 = a \alpha x + 1, \quad x \in \R.$$
From \eqref{xx} we obtain
$$a^2 (\alpha x)^2 = \phi(x,x) = \alpha x^2, \quad x \in \R.$$
We get $a^2 = 1/ \alpha$, so $\alpha > 0$ and $a = \pm \frac{1}{\sqrt{\alpha}}$. After substitution to the equation for $f$ we get $f(x) = \pm \sqrt{\alpha} x + 1$.
Conversely, it is easy to check that both such mappings solve \eqref{main0}.
\end{proof}

Our last corollary is a complex counterpart of Corollary \ref{c3}.

\begin{cor}\label{c4}
Assume that $\alpha \in \C\setminus\{0\}$ and $f\colon \C\to\C$ solves 
\begin{equation}\label{main1}
f(x+y) = f(x)f(y) - \alpha x y, \quad x, y \in \C,
\end{equation}
 Then either $f(x) = 1 +w_1x$, or $f(x) = 1 + w_2x$ for $x \in \C$, where $w_1, w_2$ are two complex roots of the second order of $\alpha$. Conversely, both mappings solve \eqref{main1}.
\end{cor}
\begin{proof}
Assume that $\K=\C$, $X = \C$ and $\phi(x,y) \coloneqq \alpha xy$. By repeating steps from the previous proof (this time without assuming $\alpha > 0$), we obtain demanded results.
\end{proof}

\section{Examples and final remarks}

We observe that Theorem \ref{t1} generally works only in one direction, that is, the converse implications do not necessarily hold.

\begin{ex}
Let $X$ be a real inner product space of dimension at least 2 and define $\phi \coloneqq \langle \cdot , \cdot \rangle$. Then, Theorem \ref{t1} implies that the potential solutions $f\colon X \to \R$ of \eqref{main} are of the form
$$f(x) =\langle x , \xi \rangle  + 1, \quad x \in X,$$
with some $\xi\in X$.
One can check that such mapping solves \eqref{main} 
if and only if $$\langle x , \xi \rangle\langle y , \xi \rangle= \langle x , y \rangle , \quad x, y \in X$$
which is impossible
if $\dim X\geq 2$. 
%Indeed, if  $f=\pm\|\cdot\|+1$ would solve \eqref{main}, then the following identity would be true:
%$$
% \|x\|+\|y\| - \|x+y\|  =  \langle x , y \rangle \pm  \|x\|\|y\|
%$$ 
%for all $x, y \in X$, which is impossible in neither case since $\dim X\geq 2$.% (e.g replace $x$ by $2x$ and choose $y$ not parallel to $x$).
\end{ex}

It may be suspected that in higher dimensions, there are no solutions to \eqref{main}. However, the following example demonstrates that this is not the case.

\begin{ex}
Let $X$ be a complex linear space and $A\colon X \to \C$ an additive nonzero functional. 
Define $\phi (x,y)\coloneqq -A(x)\cdot A(y)$ for $x, y \in X$. Then, according to Theorem \ref{t1} every solution $f\colon X \to \C$ of \eqref{main} is of the form:
$$f(x) = \gamma A(x) + 1, \quad x \in X$$
with some constant $\gamma \in \C$. A direct calculation shows that $f$ is indeed a solution if and only if $\gamma  = \pm i$.

We can choose $A$ in such a way that $A(x) \neq 0$ whenever $x \neq 0$, or such that $A$ has a bigger set of zeros. Therefore, for every complex linear space $X$ there is an abundance of nontrivial solutions $(f, \phi)$ to \eqref{main}.

This example also illustrates that the assertion of Corollary \ref{cR} does not hold in the case of complex spaces, even when the values of $\phi$ are real (since $A$ may only attain real values, as it does not necessarily have to be linear).
\end{ex}

A counterpart of the above example that works in both cases, real and complex, is also possible.

\begin{ex}
Let $X$ be a  linear space over the field $\K$ and $A\colon X \to \K$ an additive nonzero functional. 
Define $\phi (x,y)\coloneqq A(x)\cdot A(y)$ for $x, y \in X$. Then, similarly every solution $f\colon X \to \K$ of \eqref{main} is of the form:
$$f(x) = \delta A(x) + 1, \quad x \in X$$
with some constant $\delta \in \K$. further, $f$ is indeed a solution if and only if $\delta = \pm 1$.
\end{ex}

It is clear that the point $z_0$ mentioned of in condition \eqref{z0} is not unique. The above example demonstrates that many different situations are possible. Therefore, one can ask about additional properties of the set of points that satisfy \eqref{z0}. 
%Let us introduce a  relation of orthogonality on $X$ induced by $\phi$:
%$$x \perp y :\Longleftrightarrow \phi(x,y) = 0, \quad x, y \in X.$$
Define
$$Z_0\coloneqq \{z_0 \in X : \phi(z_0,z_0)\neq 0  \}.$$

We have a few observations regarding the two notions discussed above. First, a consequence of Theorem \ref{t1}, more precisely of the formula \eqref{xx}, is that the set $Z_0$ contains no pair of vectors $x, y \in X$ such that $\phi(x,y)=0$. Second, as noted in Remark \ref{r1} it follows that $Z_0 = X \setminus \ker F$, where $F\colon X \to \K$ is an additive functional.

%\begin{prop}
%Assume that $\phi$ and $f$ solves \eqref{main} and $z_0 \in Z_0$. Then
%$$\forall_{x \in X}[ x\perp z_0 \Longrightarrow x \notin Z_0\, \& \,f(x) = 1 ].$$ 
%\end{prop}

%If we strengthen the assumption of biadditivity of $\phi$ to bilinearity, then we have a better description of the set $Z_0$. 

%\section*{Declarations}

%\noindent\textbf{Acknowledgements}: 

%\medskip

%\noindent\textbf{Ethical Approval.} Not applicable.

%\medskip

%\noindent\textbf{Competing interests.} The authors declare that there is no conflict of interest regarding this publication.

%\medskip

%\noindent\textbf{Authors' contributions.} 

%\medskip

%\noindent\textbf{Funding.} This research received no funding.

%\medskip

%\noindent\textbf{Data availability.} No new data were created or analyzed in this study. Data sharing does not apply to this article.

%\section{Conclusions}

%\bigskip

%\begin{ack}
%The authors would like to express his most sincere gratitude to the three anonymous reviewers for a number of valuable comments regarding the previous versions of the manuscript, which have led to the essential improvement of the whole paper. 
%\end{ack}

\end{document}